\long\global\def\C#1\F{{}}
\documentclass[11pt]{article}
\usepackage{amsmath, amsthm, amssymb, graphicx, color, verbatim}
\usepackage[round]{natbib}
\swapnumbers

\newtheorem{theorem}{Theorem}
\newtheorem{lemma}[theorem]{Lemma}
\newtheorem{definition}[theorem]{Definition}
\newtheorem{remark}[theorem]{Remark}
\newtheorem{remarks}[theorem]{Remarks}
\newtheorem{example}[theorem]{Example}
\newtheorem{examples}[theorem]{Examples}

\newtheorem{gauss graph}[theorem]{The Graph of Gauss's Puzzle}
\newtheorem{arrang hyper}[theorem]{An Arrangement of Hyperplanes}

\newtheorem{note}[theorem]{Note}
\newtheorem{notes}[theorem]{Notes}

\newtheorem{returnpytha}[theorem]{Return to the Pythagorean Theorem}
\let\slantedreturnpytha\returnpytha
\def\returnpytha{\slantedreturnpytha\rm}
\newtheorem{orderinvariance}[theorem]{Order-invariance}
\let\slantedorderinvariance\orderinvariance
\def\orderinvariance{\slantedorderinvariance\rm}
\newtheorem{openproblem}[theorem]{Open Problem}
\let\slantedopenproblem\openproblem
\def\openproblem{\slantedopenproblem\rm}

\newtheorem{projectiontheorem}[theorem]{Projection Theorem}

\renewenvironment{proof}{{\sc Proof.}}{\EOP\wl}

\def\vtl{\vskip 1mm}
\def\tl{\vskip 2mm}
\def\wl{\vskip 4mm}

\def\es{\varnothing}
\def\][{\hspace{.03cm}]\hspace{-.1cm}[}%%%interval, (half) open or closed

\def\SB{\subseteq}

\def\Re+{\mathbb R_+}
\def\Re++{\mathbb R_{++}}

\def\EQ{\Longleftrightarrow}

\def\EOP{\phantom{a}\hfill $\square$}
\def\bitup{\bigtriangleup}

\def\FFF{{\cal F}}

\def\GGG{{\cal G}}
\def\HHH{{\cal H}}

\def\KKK{{\cal K}}

\def\LLL{{\cal L}}

\def\begeq{\begin{equation}}
\def\edeq{\end{equation}}
\def\st{{\;\vrule height8pt width0.7pt depth2.5pt\;}}

\def\sqr#1#2{{\vcenter{\vbox{\hrule height.#2pt
          \hbox{\vrule width.#2pt height#1pt \kern#1pt
              \vrule width.#2pt}
           \hrule height.#2pt}}}}
\def\square{\mathchoice\sqr56\sqr56\sqr{2.1}3\sqr{1.5}3}
\def\roster{\begin{enumerate}}
\def\endroster{\end{enumerate}}

\def\fp{\noindent}

\reversemarginpar

\let\slantedexample\example
\def\example{\slantedexample\rm}
\let\slantedexamples\examples
\def\examples{\slantedexamples\rm}

\let\slantedremark\remark
\def\remark{\slantedremark\rm}
\let\slantedremarks\remarks
\def\remarks{\slantedremarks\rm}
\let\slanteddefinition\definition
\def\definition{\slanteddefinition\rm}
\let\slantednote\note
\def\note{\slantednote\rm}
\let\slantednotes\notes
\def\notes{\slantednotes\rm}

\makeatletter
\long\def\@makecaption#1#2{%
  \vskip\abovecaptionskip
  \sbox\@tempboxa{\small #1: \sc #2}%
  \ifdim \wd\@tempboxa >\hsize
    \small #1: \sc #2\par
  \else
    \global \@minipagefalse
    \hb@xt@\hsize{\hfil\box\@tempboxa\hfil}%
  \fi
  \vskip\belowcaptionskip}
\makeatother

\begin{document}
\pagestyle{plain}

\setcounter{page}{1}

\title{
\vspace*{-1.2cm}
\Large Projections of a Learning Space\thanks{I am grateful to Eric Cosyn, Chris Doble, Mathieu Koppen, Hasan Uzun, and David Lemoine for their useful reactions to a previous draft of this paper.}}
\author{\normalsize Jean-Claude Falmagne\thanks{Dept.~of Cognitive Sciences,
University of California, Irvine, CA 92697. 
e-mail: jcf@uci.edu.
}\\
University of California, Irvine
}

\date{}
\maketitle

\thispagestyle{empty}

\vspace*{-1cm}
\begin{abstract}
\fp
Any proper subset $Q'$ of the domain $Q$ of a learning space $\KKK$ defines a projection of $\KKK$ on $Q'$ which is itself a learning space consistent with $\KKK$. Such a construction defines a partition of $Q$ having each of its classes either equal to $\{\es\}$, or preserving some key properties of the learning space $\KKK$, namely closure under union and  wellgradedness. If the set $Q'$ satisfies certain conditions, then each of the equivalence classes is essentially, via a trivial transformation, a learning space.
We give a direct proof of these and related facts which are instrumental in parsing large learning spaces. 
\end{abstract}
\def\KQ'{\KKK_{|Q'}}
\def\bigtrup{\bigtriangleup}
\def\Kk{\KKK_{[K]}}
\def\Kkp{\KKK_{[K]}^+}
\begin{comment}
{\small
\begin{roster}
\item[]
\begin{roster}
\item[Keywords:] knowledge spaces, artificial intelligence, antimatroids, media theory, psychometrics 
\end{roster}
\end{roster}
}
\end{comment}
\tl
\fp
This paper is dedicated to George Sperling whose curious, incisive mind rarely fails to produce the unexpected creative idea. George and I have been colleagues for the longest time in both of our careers.  The benefit has been mine.

Learning spaces, which are special cases of knowledge spaces \citep[cf.][]{doign99}, are mathematical structures designed to model the cognitive organization of a scholarly topic, such as Beginning Algebra or Chemistry 101. The definition of `learning space' is recalled in our Definition \ref{ls def}. Essentially, a learning space is a family of sets, called knowledge states, satisfying a couple of conditions. The elements of the sets are `atoms' of knowledge, such as facts or problems to be solved.  A knowledge state is a set gathering some of these atoms. Each of the knowledge states in a learning space is intended as a possible representation of some individual's competence in the topic. Embedded in a suitable stochastic framework, the concept of a learning space provides a mechanism for the assessment of knowledge in the sense that efficient questioning of a subject on a well chosen subset of atoms leads to gauge his or her knowledge state\footnote{As such, it offers an alternative to standardized testing, the theoretical basis of which is fundamentally different and based on the measurement of aptitudes \citep[see][for example]{nunna94}.}. Many aspects of these structures have been investigated and the results were reported in various publications; for a sample, see \citet{doignon:85,falmagne:88a, albert:b99, falma06}. The monograph by \citet{doign99}
contains most of the results up to that date\footnote{An extensive database on knowledge spaces, with hundreds of titles, is maintained by Cord Hockemeyer at the University of Graz: 
http://wundt.uni-graz.at/kst.php \citep[see also][]{hockemeyer:tr01a}.}.

In practice, in an educational context for example, a learning space can be quite large, sometimes numbering millions of states. The concept of a `projection' at the core of this paper provides a way of parsing such a large structure into meaningful components. 
Moreover, when the learning space concerns a scholarly curriculum such as  high school algebra, a projection may provide a convenient instrument for the programming of a placement test. For the complete algebra curriculum comprising several hundred types of problems, a placement test of a few dozens problems can be manufactured automatically via a well chosen projection.

The key idea is that if $\KKK$ is a learning space on a domain $Q$, then any subset $Q'$ of $Q$ defines a learning space $\KQ'$ on $Q'$ which is consistent with $\KKK$.  We call $\KQ'$ a `projection' of $\KKK$ on $Q'$, a terminology consistent with that used by \citet{cavag08} and \citet{eppst07b} for media. Moreover, this construction defines a partition of $\KKK$ such that each equivalence class is a subfamily of $\KKK$ satisfying some of the key properties of a learning space. In fact,  $Q'$ can be chosen so that each of these equivalence classes is essentially (via a trivial transformation) either a learning space consistent with $\KKK$ or the singleton $\{\es\}$.

These results, entitled `Projection Theorems' (\ref{projectiontheorem1} and \ref{projectiontheorem2}),  are formulated in this paper. They could be derived from corresponding results for the projections of media \citep[][]{cavag08, falma02, eppst07b}. Direct proof are given here. This  paper extends previous results from \citet[][Theorem 1.16 and Definition 1.17]{doign99} and \citet{cosyn:02}.

\section*{Basic Concepts}
\begin{definition}\label{ls def}  We denote by $K \bitup L = (K\setminus L) \cup (L\setminus K)$ the symmetric difference between two sets $K$, $L$, and by $d(K,L) = |K\bitup L|$ the symmetric difference distance between these sets. (All the sets considered in this chapter are finite.) The symbols ``$+$ and $\subset$ stand for the disjoint union  and the proper inclusion of sets  respectively. A {\sl (knowledge) structure} is a pair
$(Q,\KKK)$ where $Q$ is a non empty set and $\KKK$ is a family of subsets of $Q$ containing  $\es$ and $Q=\cup\KKK$. The latter is called the  {\sl domain} of $(Q,\KKK)$. The elements of $Q$ are called \emph{items} and the sets in $\KKK$ are  {\sl (knowledge) states}. Since $Q =\cup \KKK$,  the set $Q$ is implicitly defined  by $\KKK$ and we can without ambiguity call $\KKK$ a knowledge structure. A knowledge structure $\KKK$  is {\sl well--graded} if for any two distinct states $K$ and $L$ with $|K\bigtriangleup L| = n$ there exists a sequence
$
K_0=K,K_1,\ldots, K_n= L
$ 
such that $d(K,L) = n$ and $d(K_i, K_{i+1})=1$ for $0\leq i\leq n-1$. 
We call such a sequence $K_0=K,K_1,\ldots, K_n= L$ a {\sl tight path} from $K$ to~$L$. 
We say that $\KKK$ is a {\sl knowledge space} if it is closed under union, or \emph{$\cup$-closed}. 

A knowledge structure $(Q,\KKK)$ is a {\sl learning space} \citep[cf.][]{cosyn05} if it satisfies the following two conditions: 
\vspace{-.2cm}
 \roster
\item[{[L1]}] {\sc Learning smoothness.} For any two $K,L\in\KKK$, with $K \subset L$ and $|L\setminus K| = n$, there is a chain 
$
K_0=K \subset  K_1\subset\cdots\subset K_n = L
$  
such that, for $0\leq i\leq n-1$, we have $K_{i+1} = K_{i} + \{q_i\}\in\KKK$ for some $q_i\in Q$ . 
\vspace{-.2cm}
\item[{[L2]}] {\sc Learning consistency.} If $K \subset L$ are two sets in $\KKK$ such that $K + \{q\} \in \KKK$ for some $q\in Q$, then  $L \cup \{q\} \in \KKK$.  
\endroster
  A learning space is also known in the combinatorics literature as an `antimatroid', a structure introduced independently by \citet{edelm85} with slightly different, but equivalent axioms \citep[cf.~also][]{welsh95, bjorn99}. Another name is `well--graded knowledge space' \citep{falmagne:88b}; see our Lemma~\ref{echan}. 

A family  $\FFF$ of subsets of a set $Q$ is a \emph{partial knowledge structure} if it contains the set $Q=\cup\FFF$. We do not assume that $|\FFF|\geq 2$.  
We also call `states' the sets in $\FFF$.  A partial knowledge structure  $\FFF$ is a \emph{partial learning space}  if it satisfies Axioms [L1] and [L2]. Note that  $\{\es\}$ is vacuously well-graded and vacuously  $\cup$-closed, with $\cup\{\es\}=\es$. Thus, it  a partial knowledge structure  and a partial learning space (cf.~Lemma \ref{partialechan}).
\end{definition}
The following preparatory result will be helpful in shortening some proofs.
\begin{lemma}\label{toprovewg} A $\cup$-closed family of set $\KKK$ is well--graded if, for any two sets $K\subset L$, there is   a tight path from $K$ to $L$.
\end{lemma}
\begin{proof} Suppose that the condition holds. For any two distinct sets $K$ and $L$, there exists a tight path $K_0=K\subset K_1\subset  \ldots \subset K_n= K\cup L$ and another tight path $L_0=L\subset L_1\subset\ldots\subset  L_m= K\cup L$. These two tight paths can be concatenated. Reversing the order of the sets in the latter tight path and redefining 
$K_{n+1} = L_{m-1}, K_{n+2} = L_{m-2},\ldots, K_{n+m} = L_0=L$ we get the tight  path $K_0=K, K_1,\ldots, K_{n+m} = L_0 = L$, with $|K\bigtriangleup L| = n + m$.
\end{proof}

\section*{Projections}
As mentioned in our introduction, some knowledge structures may be so large that a splitting is required, for convenient storage in a computer's memory for example. Also, in some situations, only a representative part of a large knowledge structure may be needed. The concept of a projection is of critical importance in this respect. We introduce a tool for its construction.

\begin{definition}\label{def sim}   
Let $(Q,\KKK)$ be a partial knowledge structure with $|Q|\geq 2$  and let $Q'$ be any proper non empty subset of~$Q$. 
Define a relation $\sim_{Q'}$ on $\KKK$ by
\begin{align}\label{def sim eq}
K \sim_{Q'} L \quad&\EQ\quad K\cap Q' =  L \cap Q'\\
\label{def sim eq 2}
&\EQ\quad K\bitup L \SB Q\setminus Q'.
\end{align}
Thus, $\sim_{Q'}$ is an equivalence relation on $\KKK$.  When the context specifies the subset $Q'$,  we sometimes use the shorthand $\sim$ for $\sim_{Q'}$ in the sequel. The equivalence between the right hand sides of (\ref{def sim eq}) and (\ref{def sim eq 2}) is easily verified. We denote by $[K]$ the equivalence class of $\sim$ containing $K$, and by $\KKK_\sim=\{[K]\st K\in\KKK\}$ the partition of $\KKK$ induced by $\sim$. 
We may say for short that such a partition is induced by the set $Q'$. In the sequel we always assume that $|Q|\geq 2$, so that $|Q'|\geq 1$.
\end{definition}

\begin{definition}\label{def projection} Let $(Q, \KKK)$ be a partial knowledge structure  and take any non empty proper subset $Q'$ of $Q$.  The family
\begin{equation}\label{eq KQ'}
\KQ'= \{W\SB Q\st W = K\cap Q'\text{ for some }K\in\KKK\}
\end{equation}
is called the {\sl projection} of $\KKK$ on $Q'$.  We have thus
$
\KQ'\SB 2^{Q'}.
$ 
As shown by Example \ref{examp projection}, the sets in $\KQ'$ may not be states of $\KKK$.  
For any
state $K$ in $\KKK$ and with $[K]$ as in Definition~\ref{def sim}, we define  the family
\begin{equation}
\label{def Kk}
\KKK_{[K]}= \{M\st  M= L\setminus \cap[K]\text{ for some }L\sim K\}.
\end{equation}
(If $\es\in\KKK$, we have thus $\KKK_{[\es]}=[\es]$.) 
The family $\Kk$ is called a \emph{$Q'$-child}, or simply a {\sl child} of $\KKK$ (\emph{induced by $Q'$}). As shown by the example below, a child of $\KKK$ may take the form of the singleton $\{\es\}$ and we may have $\KKK_{[K]} = \KKK_{[L]}$ even when $K\not\sim L$. The set $\{\es\}$ is called the \emph{trivial} child.
\end{definition}

\begin{example}\label{examp projection} Equation (\ref{lin_space_ex}) defines a learning space $\FFF$ on the domain $Q=\{a,b,c,d,e,f\}$:
\begin{gather}\nonumber
\FFF=\{\es,\{b\},\{c\}, \{a,b\},\{a,c\},\{b,c\},\{b,d\}, \{a,b,c\},\{a,b,d\},\{b,c,d\},\{b,c,e\},\{b,d,f\}, \\
\nonumber
\{a,b,c,d\},\{a,b,c,e\},\{b,c,d,e\},\{b,c,d,f\},\{b,c,e,f\},\{a,b,d,f\}, \{a,b,c,d,e\},\\
\label{lin_space_ex} 
\hspace{-.3cm}
\{a,b,c,d,f\}, \{a,b,c,e,f\},\{b,c,d,e,f\},\{a,b,c,d,e,f\}, \{a,b,c,d,e,f,g\}\}.
\end{gather}

\begin{figure}[h]
\vspace{-4cm}
\hspace{.5cm}
\includegraphics[scale= .67]{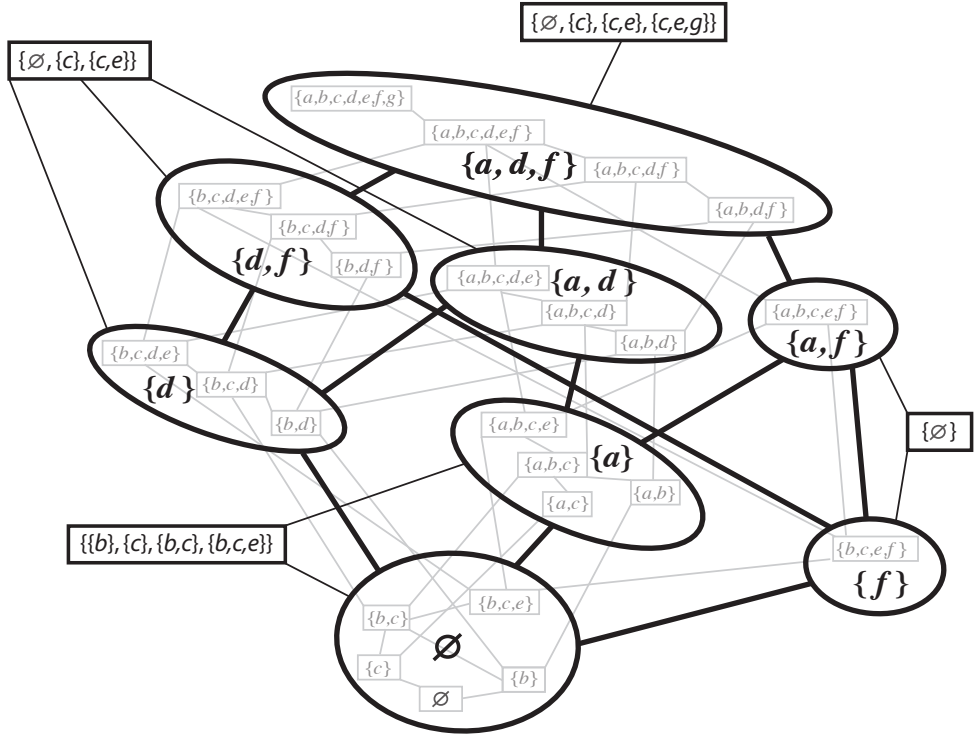}
\vskip -4.5cm
\caption{\rm\small In grey, the inclusion graph of the learning space $\FFF$ of Equation (\ref{lin_space_ex}).  Each oval surrounds an equivalence class $[K]$ (in grey) and a particular state (in black) of the projection $\FFF_{|\{a,d,f\}}$ of $Q$ on $Q'= \{a,d,f\}$, signaling a 1-1 mapping 
$\FFF_\sim\to \FFF_{|\{a,d,f\}}$ (cf.~Lemma \ref{proj def k struct =>}(ii)). Via the defining equation (\ref{def Kk}), the eight equivalence classes produce four children of $\FFF$, which are represented in the black rectangles of the figure. One of these children is 
the singleton set $\{\es\}$ (thus, a trivial child), and the others are learning spaces or partial learning spaces (cf.~Projection Theorems \ref{projectiontheorem1} and \ref{projectiontheorem2}).}
\label{projection_graph}
\end{figure}
 The inclusion graph of this learning space  is pictured by the grey parts of the diagram of Figure \ref{projection_graph}.
The sets marked in black in the eight ovals of the figure represents the states of the projection $\FFF_{|\{a,d,f\}}$  of $\FFF$ on the set $\{a,d,f\}$. It is clear that $\FFF_{|\{a,d,f\}}$ is a learning space\footnote{In fact, $\FFF_{|\{a,d,f\}} = 2^{\{a,d,f\}}$ in this particular case. This property does hold in general.}. Each of these ovals also surrounds the inclusion subgraph corresponding to an equivalence class of the 
partition~$\FFF_\sim$. 
This is consistent with  Lemma \ref{proj def k struct =>} (ii) according to which there is a 1-1 correspondence between $\FFF_\sim$ and $\FFF_{|\{a,d,f\}}$. In this example, the `learning space' property is hereditary in sense that not only  is $\FFF_{|\{a,d,f\}}$  a learning space, but also any child of $\FFF$   is a learning space or a partial learning space. Indeed, we have
 \begin{gather*}
\FFF_{[\{b,c,e\}]}= \FFF_{[\{a, b,c,e\}]} =  \{\{b\},\{c\},\{b,c\},\{b,c,e\}\},\\
\FFF_{[\{ b,c,d,e\}]}= \FFF_{[\{ b,c,d,e,f\}]} =  \FFF_{[\{a, b,c,d,e\}]} = \{\es, \{c\},\{c,e\}\},\\
 \FFF_{[\{ a,b,c,d,e,f,g\}]} = \{\es, \{c\},\{c,e\},\{c,e,g\}\}\\
  \FFF_{[\{ b,c,e,f\}]} =  \FFF_{[\{a, b,c,e,f\}]} = \{\es\}.
\end{gather*}
These four children are represented in the four black rectangles in Figure  \ref{projection_graph}. 

Theorem \ref{projectiontheorem1} shows that this hereditary property is general in the sense that the children of a partial learning space are always partial learning spaces. In the particular case of this example, just adding the set $\{\es\}$ to the child not containing it already, that is,  to the child $\FFF_{[\{b,c,e\}]}=\FFF_{[\{a, b,c,e\}]}$, would result in having all the children being learning spaces or trivial. This is {\bf not} generally true. The situation is clarified by Theorem \ref{projectiontheorem2}. 
\end{example}
\begin{remark}\label{rem project} The concept of projection for learning spaces is closely related to the concept bearing the same name for media introduced by \citet{cavag08}.  The Projection Theorems \ref{projectiontheorem1} and \ref{projectiontheorem2}, the main results of this chapter, could be derived via similar results concerning the projections of media \citep[cf.~Theorem 2.11.6 in][]{eppst07b}. This would be a detour, however. The route followed here is direct. 
\end{remark}

In the next two lemmas, we derive a few consequences of Definition \ref{def projection}. 

\begin{lemma}\label{proj def k struct =>} 
The following two statements are true for any partial knowledge structure  $(Q,\KKK)$.
\vspace{-1cm}
\begin{roster}
\item [{\rm (i)}] The projection $\KQ'$, with $Q'\subset Q$,  is a partial knowledge structure. If $(Q,\KKK)$ is a knowledge structure, then so is $\KQ'$.
\vspace{-.2cm}
\item[{\rm (ii)}]  The function $h:[K]\mapsto K\cap Q'$ is a well defined
bijection of $\KKK_\sim$ onto $\KQ'$.
\end{roster}
\end{lemma}
\begin{proof} (i) Both statements stem from the observations that $\es\cap Q'=\es$ and $Q\cap Q' = Q'$.
 \vtl
 
 (ii) That $h$ is a well defined function is due  to (\ref{def sim eq}). It is clear that $h(\KKK_\sim) = \KQ'$ by the definitions of $h$ and $\KQ'$. Suppose that, for some $[K],[L]\in\KKK_{\sim}$, we have 
$h([K]) = K \cap Q' = h([L]) = L\cap Q' = X$. 
Whether or not $X = \es$, this  entails $K \sim L$  and so $[K]=[L]$.
\end{proof}

\newpage
\begin{lemma}\label{proj def k space =>} If $\KKK$ is a $\cup$-closed family, then the following three statements are true.\vspace{-.2cm}
\begin{roster}
\item[{\rm (i)}] $K\sim \cup[K]$ for any $K\in \KKK$.
\vspace{-.2cm}
\item[{\rm (ii)}]  $\KQ'$ is a $\cup$-closed family. If $\KKK$ is a knowledge space, so is $\KQ'$.
\vspace{-.2cm}
\item[{\rm (iii)}]  The children of $\KKK$ are also $\cup$-closed.
\end{roster}
\vspace{-.2cm}
In {\rm (ii)} and {\rm (iii)}, the converse implications are not true.
\vspace{-.2cm}
\end{lemma}
\tl
For knowledge spaces, Lemma \ref{proj def k space =>} (ii) was obtained by \citet[][Theorem 1.16 on p.~25]{doign99} where the concept of a projection was referred to as a `substructure.' Their proof applies here. We include it for completeness.
\tl

\begin{proof}
(i) As $\cup[K]$ is the union of states of $\KKK$,  we get $\cup[K]\in\KKK$. We must have $K\cap Q' = (\cup[K])\cap Q'$ because 
$K\cap Q' = L\cap Q'$ for all $L\in [K]$; so $K\sim\cup[K]$.
\vtl

(ii) Since $\cup\KKK \in\KKK$ by hypothesis, $\KKK$ is a knowledge structure. Lemma \ref{proj def k struct =>}(i), implies that $\KQ'$ is a partial knowledge structure.  Any subfamily $\HHH\SB\KQ'$ is associated to the family 
$\HHH'=\{H'\in \KKK\st  H= H'\cap Q'\text{ for some }H \in \HHH\}$. As $\KKK$ is a partial knowledge space, we get $\cup \HHH'  \in\KKK$, yielding $Q'\cap( \cup \HHH')\in \KQ'$,  with
$$ 
Q' \cap \left (\cup \HHH ' \right)  = \cup_{H'\in\HHH' }(H' \cap Q') = \cup \HHH . 
$$
Thus $\KQ'$ is a partial knowledge space. This argument is valid for knowledge spaces.  
\vtl

(iii) Take $K\in\KKK$ arbitrarily. We must show that $\Kk$ is $\cup$-closed. If $\Kk = \{\es\}$, this is vacuously true. Otherwise, for any $\HHH\SB \Kk$ we define the associated family 
$$
\HHH'=\{H'\in\KKK\st H'\sim K, \, H'\setminus \cap[K] \in \HHH\}.
$$
So, $\HHH'\SB [K]$, which gives $L\cap Q'= K\cap Q'$ for any $L\in\HHH'$.  We get thus $\cup\HHH'\sim K$.
\vtl

Since $\KKK$ is $\cup$-closed, we have $\cup \HHH'\in\KKK$.   The $\cup$-closure of $\Kk$ follows from the string of equalities 
\begin{align*}
\cup\HHH = \cup_{H'\in\HHH'}(H'\setminus \cap[K] )
= \cup_{H'\in\HHH'}(H'\cap (\overline{\cap[K]})
= ( \cup_{H'\in\HHH'}H')\setminus \cap[K]
\end{align*}
which gives $\cup\HHH \in\Kk$ because  $K\sim\cup\HHH'\in\KKK$.

\vtl
Example~\ref{onehD} shows that the reverse implications in (ii) and (iii) do not hold.
\end{proof}
\newpage
\begin{example}\label{onehD} Consider the projection  of the knowledge structure
$$
\GGG = \big\{\, \es,\,\{a\},\,\{b\},\, \{c\},\, \{a,b\},\,
\{a,c\},\, \{a,b,c\} \big\},
$$
on the subset $\{c\}$. We have thus the two equivalence classes $[\{a,b\}]$ and $[\{a,b,c\}]$, with the projection   $\GGG_{|\{c\}} = \{\es,\{c\}\}$. 
The two $\{c\}$-children are
$\GGG_{[\es]}= \big\{\es,\{a\}, \{b\},\{a,b\}\big\}$ and $\GGG_{[\{c\}]} = \big\{\, \es,\, \{a\}, \{a,b\} \,\big\}$. 
Both  $\GGG_{[\es]}$ and $\GGG_{[\{c\}]}$ are well-graded and $\cup$-closed, and so is $\GGG_{|\{c\}}$, but $\GGG$ is not since $\{b,c\}$ is not a state. 
\end{example}

We omit the proof of the next result which is due to   \citet{cosyn05} .

\begin{lemma}\label{echan} A  knowledge structure $(Q,\KKK)$ is a   learning space if and only if it a well--graded  knowledge space.  
\end{lemma}
As indicated by the next lemma, the equivalence ceases to hold in the case of  partial spaces.
\begin{lemma}\label{partialechan} Any well-graded $\cup$-closed family is a partial
 learning space. The converse implication is false.
\end{lemma}
\vtl

\begin{proof} Let $\KKK$ be a well-graded $\cup$-closed family. Axiom [L1] is a special case of the well-gradedness condition.  If $K\subset L$ for two sets $K$ and $L$ in $\KKK$ and $K+\{q\}$ is in $\KKK$,  then the set $(K+\{q\})\cup L = L\cup\{q\}$ is in $\KKK$ by $\cup$--closure, and so [L2] holds. The example below disproves the converse.
 \end{proof}
\begin{example}\label{l-space(not=>)U-closed} The family of sets
\begin{align*}
\LLL = \{\{a,b,c\},\{c,d,e\}, \{a,b,c,f\},&\{c,d,e,g\},\{a,b,c,f,d\},\{c,d,e,g,b\},\\
&\{a,b,c,f,d,e\},\{c,d,e,g,b,a\},\{a,b,c,d,e,f,g\}\}
\end{align*}
is a partial learning space since it is the union of the two chains
\begin{gather*}
\{a,b,c\} \subset\{a,b,c,f\}\subset\{a,b,c,f,d\}\subset
\{a,b,c,f,d,e\}\subset\{a,b,c,d,e,f,g\},\\
\{c,d,e\} \subset \{c,d,e,g\} \subset\{c,d,e,g,b\} \subset
\{c,d,e,g,b,a\} \subset\{a,b,c,d,e,f,g\}
\end{gather*}
with the only common set $\cup\LLL$. However, $\LLL$ is neither $\cup$-closed nor well-graded.
\end{example}

We state the first of our two projection theorems.

\newpage
 \begin{projectiontheorem}\label{projectiontheorem1}
Let $\KKK$ be a learning space (resp.~a well-graded $\cup$-closed family) on a domain $Q$  with $|Q|=|\cup\KKK|\geq 2$. The following two properties hold for any proper non empty subset $Q'$ of $Q$:
\vspace{-.2cm}
\begin{roster}
\item[{\rm (i)}] The projection $\KKK_{|Q'}$ of $\KKK$ on $Q'$ is a learning space (resp.~a well-graded $\cup$-closed family);\vspace{-.3cm}
\item[{\rm (ii)}] In either case, the  children of $\KKK$ are well-graded and  $\cup$-closed families.
\end{roster}
\end{projectiontheorem}

Note that we may have $\Kk = \{\es\}$ in (ii) (cf.~Example \ref{examp projection}).

\tl
\begin{proof}
(i) If $\KKK$ is a learning space, then $\KQ'$ is a knowledge structure by Lemma \ref{proj def k struct =>}(i). Since $\KKK$ is $\cup$-closed by Lemma \ref{echan}, so is $\Kk$ by Lemma   \ref{proj def k space =>}  (ii).
 It remains to show that $\Kk$ is well-graded.  (By Lemma \ref{echan} again, this will imply that $\KQ'$ is a learning space.)  We use use Lemma  \ref{toprovewg} for this purpose. Take any two states $K'\subset L'$ in $\KQ'$ with $d(K',L') = n$ for some positive integer $n$. 
By 
Lemma \ref{proj def k struct =>} (ii), we have thus $K' = K\cap Q'$ and $L' = L\cap Q'$ for some $K,L\in \KKK$. As $\KKK$ well-graded by Lemma \ref{echan},  there exists a tight path 
$
K_0=K, K_1, \ldots , K_m= L
$
from $K$ to $L$, with either $ K_j= K_{j-1} + \{p_j\}$ or  $K_{j-1}= K_{j} + \{p_j\}$ for some $p_j\in Q$ and $1\leq i\leq m$.  Let $j\in\{1,\ldots, m\}$ be the first index such that $p_j\in Q'$. We have then necessarily  $p_j\in Q'\cap L = L'$ and $ K_j= K_{j-1} + \{p_j\}$. This yields
\begin{gather*}
K' = K_0\cap Q'= K_1\cap Q'=\ldots=K_{j-1}\cap Q',\\
\noalign{and for $1\leq j \leq n$,}
K_j \cap Q' = (K_{j-1}\cap Q') + \{p\},
\end{gather*}
with $ p\in L'\setminus K'$. Defining $ K'_0 = K'$ and $ K'_1 = K_j\cap Q'$,
we have $d(K'_1, L') = n-1$, with $K'_1$ a state of $\KQ'$.
By induction, we conclude that $\KQ'$ a wellgraded $\cup$ closed knowledge structure. Since $\KQ'$ is $\cup$-closed, it must be a learning space by Lemma \ref{echan}.
\vtl

Suppose now that $\KKK$ is a well-graded $\cup$-closed family (rather than learning space). In such a case, there is no need to invoke Lemma \ref{echan} and the above argument can be used to prove that $\KQ'$ a wellgraded $\cup$ closed family. 
\tl

(ii) Take any child $\Kk$ of $\KKK$. By Lemma \ref{proj def k space =>}(iii), $\Kk$ is a $\cup$-closed family. We use Lemma~\ref{toprovewg} to prove that $\Kk$ is also well-graded. Take any two states  $M\subset L$ in $\Kk$. We have thus   $L = L'\setminus(\cap[K])$ and $M=M'\setminus(\cap[K])$ for some $L'$ and $M'$ in $[K]$, with 
\begin{equation}\label{capKSBL}
\cap[K]\SB L'\subset M'.
\end{equation}
Since $\KKK$ is well-graded, there is a tight path
\begin{equation}\label{tightpathchain}
L'_0 = L'\subset L'_1\subset\ldots\subset L'_n= M'
\end{equation}
with all its states in $[K]$. Indeed, $L'\subset L'_j\subset M'$ and $L'\cap Q'= M'\cap Q'$ imply $  L'\cap Q'=  L'_j\cap Q'= M'\cap Q'$ for any index $1\leq j\leq n-1$. We now define the sequence $L_j = L'_j\setminus\cap[K]$, $0\leq j \leq n$. It is clear that (\ref{capKSBL}) and (\ref{tightpathchain}) imply
$$
L_0 = L\subset L_1\subset\ldots\subset L_n= M,
$$
and it is easily verified that $L_0 = L, L_1,\ldots, L_n= M$ is a tight path from $L$ to $M$. Applying Lemma \ref{toprovewg}, we conclude that $\Kk$ is well-graded. 
\end{proof}

\begin{remark}\label{assinges} In Example \ref{examp projection}, we had a situation in which the non trivial children of a learning space were either themselves learning spaces, or would become so by the addition of the set~$\{\es\}$.  This can happen if and only if the subset $Q'$ of the domain defining the projection satisfies the condition spelled out in the next definition.
\end{remark}

 \begin{definition}\label{yielding} Suppose that $(Q,\KKK)$ is a partial knowledge structure, with $|Q|\geq 2$. A subset $Q'\subset Q$ is \emph{yielding} if for any  state $L$ of $\KKK$ that is minimal for inclusion in some equivalence class $[K]$, we have  $|L\setminus \cap[K] |\leq 1$.  We recall that $[K]$ is the equivalence class containing $K$ in the partition of $\KKK$  induced by $Q'$ (cf.~Definition \ref{def sim}). For any non trivial child $\Kk$ of $\KKK$, we call $\Kk^+=\Kk \cup \{\es\}$ a \emph{plus} child of $\KKK$.
 \end{definition}

\begin{projectiontheorem}\label{projectiontheorem2} Suppose that $(Q,\KKK)$ is a learning space  with $|Q|\geq 2$, and let $Q'$ be a proper non empty subset of $Q$. The two following conditions are then equivalent.
 
{\rm (i)} The set $Q'$ is yielding.

{\rm (ii)} All the plus children of $\KKK$ are learning space\footnote{Note that we may have $\{\es\}\in\Kk$, in which case $\Kk^+ = \Kk$ (cf.~Example \ref{examp projection}).}.\end{projectiontheorem}

(It is easily shown that any learning space has always at least on non trivial child.)
\vtl

\begin{proof}
(i) $\Rightarrow$ (ii). By Lemma \ref{proj def k space =>}(iii), we know that any non trivial child $\Kk$ is\linebreak $\cup$-closed. This implies that the associated plus child $\Kkp$ is a knowledge space. We use 
Lemma \ref{toprovewg} to prove that $\Kkp$ is also well--graded. Suppose that $L$ and $M$ are states of $\Kkp$, with
$\es \SB L\subset M$ and, say $d(L,M) = n$. 

{\sc Case 1.} Suppose that $L\neq \es$.  By the definition of $\Kkp$, we have $L = L'\setminus \cap[K]$ and $M = M'\setminus \cap[K]$ for some $L',M'\in [K]$. Since $L\subset M$,
 we must have $L'\subset M'$. 
Because $L'$ and $M'$ are sets of $[K]$, there exists in $[K]$  (by the Projection Theorem \ref {projectiontheorem1} (ii)) a tight path 
$
L'_0 = L' \subset L'_1\subset \ldots\subset L'_m = M'.  
$
We show that  this tight path defines a tight path 
\begin{equation}\label{wg path 2 in (ii)}
L_0 = L, L_1, \ldots, L_n = M  
\end{equation}
lying entirely in $\Kkp$ (actually, in $\Kk$). By definition of a tight path,
we have $L'_0 + \{p_1\} = L'_1$ for some $p_1\in Q\setminus Q'$. Defining $L_1 = L'_1 \setminus\cap[K] $, we get $L_0 + \{p_1\} =L_1$ and thus $d(L_1,M) = n-1$. Note that $L_1$ is in $\Kkp$ because $L'_1$ is is $[K]$.
The existence of the tight path (\ref{wg path 2 in (ii)}) follows by induction.

{\sc Case 1.} Suppose now that $L= \es$. In view of what we just proved, we only have to show that, for any non empty $M\in\Kkp$, there is a singleton set $\{q\}\in \Kkp$ with $q\in M$. By definition of $\Kkp$, we have $M = M'\setminus\cap[K] $ for some $M'\in [K]$.  Take a minimal state $N$ in $[K]$ such that $N\SB M'$ and so 
$N\setminus\cap[K] \SB M$. Since $Q'$ is yielding, we get $|N\setminus\cap[K]| \leq 1$. If $|N\setminus\cap[K]|=1$, then $N\setminus\cap[K]=\{q\}\SB M$ for some $q \in Q$ with $\{q\}\in\Kkp$. Suppose that $|N\setminus\cap[K]| = 0$. 
 Thus $N\setminus\cap[K] = \es$ and $N$ must be the only minimal set in $[K]$, which implies that $\cap[K]= N$. By the wellgradedness of $[K]$ established in the Projection Theorem \ref{projectiontheorem1}(ii), there exists some $p\in M$ such that $N+\{p\}\SB M$. We get thus 
 $$
 (N + \{p\})\setminus \cap [K] = (N + \{p\})\setminus N = \{p\}\SB M\quad\text{with}\quad \{p\}\in\Kkp.
 $$
 The tight path (\ref{wg path 2 in (ii)}) from $L$ to $M$ exists thus in both cases. Applying Lemma \ref{toprovewg}, we can assert that $\Kkp$ is well-graded. We have shown earlier that  $\Kkp$ is a knowledge space. Accordingly, the plus child $\Kkp$ is a learning space.
\tl

(ii) $\Rightarrow$ (i). If some equivalence class $[K]$ is a chain or is a single set\footnote{We saw in Example \ref{examp projection} that this is possible.} $\{K\} = [K]$, then  
$|K\setminus \cap[K]| = 0$. Otherwise $[K]$ contains more than one minimal state. Let $L$ be one of these minimal states. Thus $L\setminus\cap[K]$ is a minimal non empty set of $\Kkp$, which by hypothesis is a learning space. By the wellgradedness of $\Kkp$, there is a tight path from $\es$ to $L\setminus\cap[K]$. Because $L\setminus\cap[K]$ is non empty and minimal in $\Kk$, it must be a singleton. We get thus $|L\setminus\cap[K]=1$.
\end{proof}

\section*{Summary}
Performing an assessment in a large learning space $(Q,\KKK)$  may be impractical in view of memory limitation or for other reasons. In such a case, a two-step or an n-step procedure may be applied. On Step 1, a representative subset $Q'$ of items from the domain $Q$ is selected, and an assessment is performed on the projection learning space $\KKK_{|Q'}$ induced by~$Q'$ (cf.~Projection Theorem \ref{projectiontheorem1}(i)). The outcome of this assessment is some knowledge state $K\cap Q'$ of $\KKK_{|Q'}$
which corresponds (1-1) to equivalence class $[K]$ of the partition of $\KKK$ induced by~$Q'$  (cf.~Lemma \ref{proj def k struct =>} (ii)).   On Step 2, the child $\Kk$ is formed by removing all the common items in the states of $[K]$. The assessment can then be pursued on $\Kk$ of $\KKK$ which is a partial learning space  (cf.~Projection Theorem \ref{projectiontheorem1}(ii)). The outcome of Step 2 is a set $L\setminus\cap[K]$, where $L$ is a state in the learning space~$\KKK$. This 2-step procedure can be expanded into a n-step recursive algorithm if necessary. 

If the set $Q'$ is yielding, then any equivalence class $[K]$containing more than one set can be made into a learning space by a trivial transformation. This property is not critical for the 2-step procedure outlined above.

\section*{References}

\begin{roster}

\bibitem[Albert and Lukas(1999)]{albert:b99}
\hspace{-1cm}D.~Albert and J.~Lukas, editors.
\newblock \emph{Knowledge Spaces: Theories, Empirical Research, Applications}.
\newblock Lawrence Erlbaum Associates, Mahwah, NJ, 1999.

\bibitem[Bj{\"o}rner et~al.(1999)Bj{\"o}rner, Las~Vergnas, Sturmfels, White,
  and Ziegler]{bjorn99}
\hspace{-1cm}A.~Bj{\"o}rner, M.~Las~Vergnas, B.~Sturmfels, N.~White, and G.M. Ziegler.
\newblock \emph{{Oriented Matroids}}.
\newblock Cambridge University Press, Cambridge, London, and New Haven, second
  edition, 1999.

\bibitem[Cavagnaro(2008)]{cavag08}
\hspace{-1cm}D.R. Cavagnaro.
\newblock {Projection of a medium}.
\newblock To be published in the \emph{Journal of Mathematical Psychology},
  2008.

\bibitem[Cosyn and Uzun(2008)]{cosyn05}
\hspace{-1cm}E.~Cosyn and H.B. Uzun.
\newblock {Axioms for learning spaces}.
\newblock Accepted for publication in the \emph{Journal of Mathematical
  Psychology}, 2008.

\bibitem[Cosyn(2002)]{cosyn:02}
\hspace{-1cm}Eric Cosyn.
\newblock Coarsening a knowledge structure.
\newblock \emph{Journal of Mathematical Psychology}, 46:\penalty0 123--139,
  2002.

\bibitem[Doignon and Falmagne(1999)]{doign99}
\hspace{-1cm}J.-P. Doignon and J.-Cl. Falmagne.
\newblock \emph{{Knowledge Spaces}}.
\newblock Springer-Verlag, Berlin, Heidelberg, and New York, 1999.

\bibitem[Doignon and Falmagne(1985)]{doignon:85}
\hspace{-1cm}J.-P. Doignon and J.-Cl. Falmagne.
\newblock {Spaces for the Assessment of Knowledge}.
\newblock \emph{International Journal of Man-Machine Studies}, 23:\penalty0
  175--196, 1985.

\bibitem[Edelman and Jamison(1985)]{edelm85}
\hspace{-1cm}P.H. Edelman and R.~Jamison.
\newblock {The theory of convex geometries}.
\newblock \emph{Geometrica Dedicata}, 19:\penalty0 247--271, 1985.

\bibitem[Eppstein et~al.(2007)Eppstein, Falmagne, and S.]{eppst07b}
\hspace{-1cm}D.~Eppstein, J.-Cl. Falmagne, and Ovchinnikov S.
\newblock \emph{{Media Theory}}.
\newblock Springer-Verlag, Berlin, Heidelberg, and New York, 2008.

\bibitem[Falmagne and Doignon(1988{\natexlab{a}})]{falmagne:88a}
\hspace{-1cm}J.-Cl. Falmagne and J-P. Doignon.
\newblock {A class of stochastic procedures for the assessment of knowledge}.
\newblock \emph{British Journal of Mathematical and Statistical Psychology},
  41:\penalty0 1--23, 1988{\natexlab{a}}.

\bibitem[Falmagne and Doignon(1988{\natexlab{b}})]{falmagne:88b}
\hspace{-1cm}J.-Cl. Falmagne and J-P. Doignon.
\newblock {A Markovian procedure for assessing the state of a system}.
\newblock \emph{Journal of Mathematical Psychology}, 32:\penalty0 232--258,
  1988{\natexlab{b}}.

\bibitem[Falmagne and Ovchinnikov(2002)]{falma02}
\hspace{-1cm}J.-Cl. Falmagne and S.~Ovchinnikov.
\newblock {Media theory}.
\newblock \emph{Discrete Applied Mathematics}, 121:\penalty0 83--101, 2002.

\bibitem[Falmagne et~al.(2006)Falmagne, Cosyn, Doignon, and
  Thi{\'e}ry]{falma06}
\hspace{-1cm}J.-Cl. Falmagne, E.~Cosyn, J.-P. Doignon, and N.~Thi{\'e}ry.
\newblock {The assessment of knowledge, in theory and in practice}.
\newblock In B.~Ganter and L.~Kwuida, editors, \emph{Formal Concept Analysis,
  4th International Conference, ICFCA 2006, Dresden, Germany, February 13--17,
  2006}, Lecture Notes in Artificial Intelligence, pages 61--79.
  Springer-Verlag, Berlin, Heidelberg, and New York, 2006.

\bibitem[Hockemeyer(2001)]{hockemeyer:tr01a}
\hspace{-1cm}C.~Hockemeyer.
\newblock {Tools and Utilities for Knowledge Spaces}.
\newblock Unpublished technical report, Institut f{\"u}r Psychologie,
  Karl--Franzens--Universit{\"a}t Graz, Austria, 2001.

\bibitem[Nunnally and Bernstein(1994)]{nunna94}
\hspace{-1cm}J.~Nunnally and I.~Bernstein.
\newblock \emph{{Psychometric Theory}}.
\newblock MacGraw-Hill, New York, 1994.

\bibitem[Welsh(1995)]{welsh95}
\hspace{-1cm}D.J.A. Welsh.
\newblock {Matroids: Fundamental Concepts}.
\newblock In R.L. Graham, M.~Gr{\"o}tschel, and L.~Lov{\'a}sz, editors,
  \emph{Handbook of Combinatorics}, volume~1. The M.I.T. Press, Cambridge, MA,
  1995.

\end{roster}
%\bibliography{Projections_LS.bib}
%\bibliographystyle{plainnat}
%\bibliographystyle{apalike}
\end{document}